\documentclass{amsart}
\usepackage{mathrsfs}
\usepackage{mathrsfs}
\usepackage{amsfonts}
\usepackage{cases}
\usepackage{latexsym}
\usepackage{amsmath}
\usepackage[all]{xy}
\usepackage{stmaryrd}
\usepackage{amsfonts}
\usepackage{amsmath,amssymb,amscd,bbm,amsthm,mathrsfs,dsfont}

\usepackage{fancyhdr}
\usepackage{amsxtra,ifthen}
\usepackage{verbatim}

\numberwithin{equation}{section}

\theoremstyle{plain}
\newtheorem{theorem}{Theorem}[section]
\newtheorem{lemma}[theorem]{Lemma}

\newtheorem{corollary}[theorem]{Corollary}

\theoremstyle{definition}

\newtheorem{example}[theorem]{Example}

\newtheorem{remarks}[theorem]{Remarks}

\begin{document}

\title[Derivations and Hochschild cohomology of zigzag algebras]
{Derivations and Hochschild cohomology of zigzag algebras}

\author{Yanbo Li and Zeren Zheng}

\address{Li: School of Mathematics and Statistics, Northeastern
University at Qinhuangdao, Qinhuangdao, 066004, P.R. China}

\email{liyanbo707@163.com}

\address{Zheng: School of Mathematics and Statistics, Northeastern
University at Qinhuangdao, Qinhuangdao, 066004, P.R. China}

\email{k1169182904@icloud.com}

\begin{abstract}
Let $\Gamma$ be a connected graph without loops, cycles or multiple edges and $Z(\Gamma)$
the corresponding zigzag algebra. Then every Jordan derivation of $Z(\Gamma)$ is a derivation.
Moreover, we will prove that the dimension of 1th Hochschild cohomology group of $Z(\Gamma)$
is one by computing the dimensions of linear spaces spanned by derivations and inner derivations.
This implies that the dimension of the 1th Hochschild cohomology group of each algebra derived equivalent to a zigzag algebra is 1.
\end{abstract}

\thanks{Corresponding Author. Zeren Zheng}



\thanks{The work is supported by the Natural Science Foundation of Hebei Province, China (A2021501002); China Scholarship Council (202008130184)  and
NSFC 11871107.}

\maketitle

\section{Introduction}

Zigzag algebras were introduced by Huerfano and Khovanov \cite{HK} in the context of categorification of the
adjoint representation of simply-laced quantum groups. Along this way, they were also connected to Heisenberg algebras, see \cite{CL, RS} for more details.
Zigzag algebras appear in various active research areas in modern mathematics, including symplectic geometry \cite{EL},
representations and blocks of symmetric groups \cite{EK}, blocks of Temperley-Lieb algebras \cite{W}, Soergel bimodules theory \cite{MT} and so on.
Furthermore, there are some generalizations of zigzag algebras.
For example, skew zigzag algebras \cite{C, HK}, higher zigzag algebras \cite{G} and affine zigzag algebras \cite{KM}.

Zigzag algebras are defined in the form of path algebras. Path algebras naturally appear in the study of tensor
algebras of bimodules over semi-simple algebras. It is well known that any
finite-dimensional basic $K$-algebra is  given by a quiver with relations when
$K$ is an algebraically closed field and consequently, many problems of representation theory of finite dimensional algebras can be transferred to path algebras.

Linear maps of associative algebras play significant
roles in various mathematical areas, such as Lie theory, representation theory, matrix theory and operator algebras.
We refer the reader to a series of papers \cite{LW1, LW2, LW3} for some results on linear maps of path algebras.
In particular, it is well known that the derivations of an associative algebra own a Lie algebra structure.
The properties of the Lie algebra can be used to study the associative algebra. Relevant results can be found in \cite{GL, Lin, LD} and the references therein.
Note that the Lie algebra used in \cite{LD} is the quotient of derivation Lie algebra by the Lie ideal of inner derivations.
As a linear space, it is isomorphic to the 1th Hochschild cohomology group, which is an invariance of derived equivalence.
The Hochschild cohomology groups are also closely related to the center and deformation theory of the given algebra.

For the theory of Hochschild cohomology, it is important to study the actual structure of the
Hochschild cohomology groups for particular classes of algebras and many papers are devoted to do it, such as \cite{BE, GL, H, PX, XHJ} and so on.
The main result of this note is to prove the dimension of 1th Hochschild cohomology group of a zigzag algebra is 1
by computing the dimensions of linear spaces spanned by derivations and inner derivations, which will be given
after determining the form of derivations of a ziagzag algebra in Section 3.

\bigskip

\section{Zigzag algebras}

In this section, we recall the definition of zigzag algebras and fix all notations that we need in the next section.
The main references are \cite{ARS, HK}.

A \textit{finite quiver} $Q$ is an oriented graph
with the set of vertices $Q_0$ and the set of arrows between
vertices $Q_1$ being both finite.
For an arrow $\alpha^{i,j}$ from vertex $i$ to vertex $j$, write $s(\alpha^{i,j})=i$ and $e(\alpha^{i,j})=j$.
We often write an arrow by $\alpha$ for simplicity if there is no danger of confusion.
A length $n$ \textit{nontrivial path} $p=(a|\alpha_1\cdots\alpha_n|b)$ in $Q$ is an ordered
sequence of arrows such that
$s(\alpha_1)=a, e(\alpha_{n})=b$ and $e(\alpha_l)=s(\alpha_{l+1})$ for each $l<n$. A \textit{trivial path}
is the symbol $e^i$ for each $i\in Q_0$. In this case, we set
$s(e^i)=e(e^i)=i$.

Let $K$ be a field and $Q$ be a quiver. Then the path algebra
$KQ$ is the $K$-algebra generated by the paths in $Q$
and the product of two paths $x=(a|\alpha_1\cdots\alpha_m|b)$ and
$y=(c|\beta_1\cdots\beta_n|d)$ is defined by
$$xy=\delta_{bc}(a|\alpha_1\cdots\alpha_m\beta_1\cdots\beta_n|d)$$
Clearly, $KQ$ is an associative algebra with the identity
$1=\sum_{i\in \Gamma_0}e^i$, where $e^i(i\in Q_0)$ are
pairwise orthogonal primitive idempotents of $KQ$.

A {\em relation} $\sigma$ on a quiver $Q$ is
a $K$-linear combination of paths $\sigma=\sum_{i=1}^nk_ip_i,$
where $k_i\in K$ and
$e(p_1)=\cdots=e(p_n),\,\,\, s(p_1)=\cdots=s(p_n).$ Moreover, the
number of arrows in each path is assumed to be at least 2. Let
$\rho$ be a set of relations on $\Gamma$. The pair
$(Q, \rho)$ is called a \textit{quiver with relations}.
Denote by $K(Q,
\rho)$ the algebra $K\Gamma/<\rho>$, where $<\rho>$ is the ideal of $K\Gamma$ generated by the set
of relations $\rho$.
For arbitrary element $x\in KQ$, write by $\overline x$ the
corresponding element in $K(Q, \rho)$. We often write
$\overline x$ as $x$ if this is not misleading or confusing.

Given a connected graph $\Gamma$, denote the set of vertices by $\Gamma_0$. Define a quiver $Q$ with $Q_0=\Gamma_0$
and $Q_1=\{\alpha^{ij}, \alpha^{ji}\mid i-j\,\, {\rm in}\,\, \Gamma$\},
where $i-j$ implies that there is a line in $\Gamma$ connecting $i$ with $j$.
Then type $\Gamma$ zigzag algebra $Z(\Gamma)$ is the path algebra of
quiver $Q$ with relations as follows:

(1) All paths of length three or greater are zero.

(2) All paths of length two that are not cycles are zero.

(3) All length-two cycles based at the same vertex are equal.

\smallskip

An example of a zigzag algebra is illustrated below.

\begin{example}
(Zigzag algebra of type A) Let $K$ be a field and $Q$ the following quiver
$$\xymatrix@C=13mm{
  \bullet \ar@<2.5pt>[r]^{\alpha_1}  & \bullet \ar@<2.5pt>[r]^(0.4){\alpha_2}
  \ar@<2.5pt>@[r][l]^(1){1}^{\alpha_1'}^(0){2}
  &\bullet\ar@<2.5pt>@[r][l]^(0.20){3}^(0.6){\alpha_2'}\cdots
  \bullet\ar@<2.5pt>[r]^(0.6){\alpha_{n-1}} & \bullet\ar@<2.5pt>@[r][l]^(0.35){\alpha_{n-1}'}^(0.75){n-1}^(0){n}\\
}$$ with relation $\rho$ given above.
\end{example}

Note that in \cite{ET}, Ehrig and Tubbenhauer gave a slightly different definition of a zigzag algebra. However, they are equivalent.

\smallskip

Denote the cycle $\alpha^{i,j}\alpha^{j,i}$ by $c^i$. Then the following lemma is easily verified.

\begin{lemma}\label{2.1}
The zigzag algebra $Z(\Gamma)$ is a finite algebra with basis
$\{e_i\mid i\in Q_0\}\cup \{c^i\mid i\in Q_0\}\cup Q_1$,
and the center $C(Z(\Gamma))$ has a basis $\{c^i\mid i\in Q_0\}\cup \{1\}$.
Moreover, $\dim Z(\Gamma)=2|Q_0|+|Q_1|$ and $\dim C(Z(\Gamma))=|Q_0|+1$.
\end{lemma}

\bigskip

\section{Derivations and Jordan derivations}

In this section, we describe the form of a derivation of zigzag algebra $Z(\Gamma)$. So let us begin with the definition of a derivation.
Let $K$ be a field and $\mathcal{A}$ a $K$-algebra. Recall that a linear mapping $\Theta$ from $\mathcal{A}$ into itself is called a
\textit{derivation} if
$$
\Theta(ab)=\Theta(a)b+a\Theta(b)
$$
for all $a, b\in \mathcal{A}$.

For simplicity, we use Einstein summation convention from now on. Note that all the places where do not mean sums are easy to know and we will not point them out later.

\begin{lemma}\label{3.1}
A linear mapping $\Theta$ is a derivation of $Z(\Gamma)$ if and only if
\begin{enumerate}
\item[\rm{(1)}] $\Theta(e^i)=t^i_{mi}\alpha^{mi}+t^i_{in}\alpha^{in}$;

\smallskip

\item[\rm{(2)}] $\Theta(\alpha^{ij})=t^{ij}_{ij}\alpha^{ij}+t_{ji}^{j}c^{i}+t_{ji}^{i}c^{j}$;

\smallskip

\item[\rm{(3)}] $\Theta(c^i)=(t^{ij}_{ij}+t^{ji}_{ji})c^i$,
\end{enumerate}
where all coefficients are in $K$ and $t^j_{ji}=-t^i_{ji}$.
\end{lemma}

\begin{proof}
Let $\Theta$ be a derivation of $Z(\Gamma)$ and assume that $$\Theta(e^i)=k_r^ie^r+t_{mn}^i\alpha^{mn}+l_r^ic^r \eqno(3.1)$$
Note that $e^i$ is an idempotent.
This implies that $$\Theta(e^i)=\Theta(e^i)e^i+e^i\Theta(e^i)\eqno(3.2)$$ and consequently,
$$e^i\Theta(e^i)e^i=0.\eqno(3.3)$$ Combining (3.1) with (3.3) shows that $$k_i^i=l_i^i=0. \eqno(3.4)$$
Moreover, substituting (3.1) into (3.2) and using (3.4) gives that (1) holds.
Note that $e^ie^j=0$ and hence $$\Theta(e^i)e^j+e^i\Theta(e^j)=0.\eqno(3.5)$$
Now substituting (1) into (3.5) yields $t^j_{ji}=-t^i_{ji}.$

In order to prove (2), suppose that $$\Theta(\alpha^{ij})=k_r^{ij}e^r+t_{mn}^{ij}\alpha^{mn}+l_r^{ij}c^r.$$ Note that $\alpha^{ij}=e^i\alpha^{ij}$ if $i\neq j$ and thus
$$
\Theta(\alpha^{ij})=\Theta(e^i\alpha^{ij})=\Theta(e^i)\alpha^{ij}+e^i\Theta(\alpha^{ij}).\eqno(3.6)
$$
By substituting (1) into (3.6) we get
$$\Theta(\alpha^{ij})=t_{ji}^ic^i+k_i^{ij}e^i+t_{in}^{ij}\alpha^{in}+l_i^{ij}c^i.\eqno(3.7)$$
Similarly, applying the fact $\alpha^{ij}=\alpha^{ij}e^j$ leads to
$$\Theta(\alpha^{ij})=t_{ji}^ic^i+k_j^{ij}e^j+t_{mj}^{ij}\alpha^{mj}+l_j^{ij}c^j.\eqno(3.8)$$
Then we complete the proof of (2) by comparing (3.7) with (3.8).
Recall that $c^i=\alpha^{ij}\alpha^{ji}$. Then (3) can be obtained from (2) by easy computation.

\smallskip

Conversely, if $\Theta$ is a linear map on $Z(\Gamma)$ satisfying the conditions (1)-(3),
then it is easy to check that $\Theta$ is a derivation. We omit the details here.
\end{proof}

\begin{remarks}
(1) Guo and Li \cite{GL} studied the form of a derivation of a path algebra of a quiver without relations, and thus their result can not be used in this note.

(2) For all $a,b\in \mathcal{A}$, denote the \textit{Jordan product} by $a\circ b=ab+ba$.
Then a linear mapping from $\mathcal{A}$ into itself is called a {\em Jordan derivation} if
$$\Theta(a\circ b)=\Theta(a)\circ b+a\circ \Theta(b).$$
Every derivation is obviously a Jordan
derivation. The converse statement is not true in
general. Moreover, an \textit{antiderivation} is a linear mapping of $\mathcal{A}$ if
$$\Theta (ab)=\Theta (b)a+b\Theta(a)$$ for all $a, b\in
\mathcal{A}$. Note that there has been an increasing interest in the study of Jordan derivations of various
algebras. The standard problem is to find out whether a Jordan derivation degenerate to a derivation.
We refer the reader to \cite{Br, Cu, He, XW, ZY} and the references therein for relevant results on this topic.

As in Lemma \ref{3.1}, we can determine the forms of anti-derivations and Jordan derivations of $Z(\Gamma)$.
Then the following results are easily verified. We omit the details and leave them to the reader.
\begin{enumerate}
\item[\rm{(i)}] Every anti-derivation of $Z(\Gamma)$ is $0$.

\item[\rm{(ii)}] Every Jordan derivation of $Z(\Gamma)$ is a derivation.
\end{enumerate}
\end{remarks}

\bigskip

\section{Hochschild cohomology of zigzag algebras}
In this section, we will compte the 1th Hochschild cohomology group of a zigzag algebra.
Denote the linear space spanned by all the derivations of $Z(\Gamma)$ by $Der(Z(\Gamma))$. We need to compute the dimension of $Der(Z(\Gamma))$.

\begin{lemma}\label{3.3}
Let $\Gamma$ be a connected finite graph without loops, cycles or multiple edges. If $|\Gamma_0|>1$, then $\dim Der (Z(\Gamma))=3|\Gamma_0|-2$.
\end{lemma}

\begin{proof}
We prove this lemma by induction on the number of vertices of $\Gamma$.

Let $|\Gamma_0|=2$. Then $Z(\Gamma)$ has a basis $\{e^1, e^2, \alpha^{12}, \alpha^{21}, c^1, c^2\}$.
For an arbitrary derivation $\Theta$ of $Z(\Gamma)$, we have from Lemma \ref{3.1} that
$$\Theta(e^1, e^2, \alpha^{12}, \alpha^{21}, c^1, c^2)=(e^1, e^2, \alpha^{12}, \alpha^{21}, c^1, c^2)A,$$
where $A$ is of the form $$\begin{bmatrix}
0 & 0 & 0 & 0 & 0 & 0\\
0 & 0 & 0 & 0 & 0 & 0\\
t_{12}^1 & -t_{12}^1 & t_{12}^{12} & 0 & 0 & 0\\
t_{21}^1 & -t_{21}^1 & 0 & t_{21}^{21} & 0 & 0\\
0 & 0 & t_{21}^2 & -t_{12}^1 & t_{12}^{12}+t_{21}^{21} & 0\\
0 & 0 & 0 & t_{12}^1 & 0 & t_{12}^{12}+t_{21}^{21}
\end{bmatrix}$$
Clearly, the space spanned by derivations of $Z(\Gamma)$ is isomorphic to the linear space spanned by the matrices $A$.
Then it is easy to check that $$\dim Der Z(\Gamma)=4=3\times 2-2,$$ that is, the lemma holds when $|\Gamma_0|=2$.

Suppose that the lemma holds for $|\Gamma_0|=n-1$. Let $\Gamma$ be a connected finite graph without loops, cycles or multiple edges and suppose that $|\Gamma_0|=n$.
Then there must exists a vertex that is connected with only one other vertex. Without lose of generality, we assume the vertex is labeled by $n$ and it is connected with vertex $n-1$. Then take a vertex being connected with vertex $n-1$ that is different from vertex $n$ and assume it is labeled by $n-2$. Denote by $\Gamma \setminus \{n\}$ the graph obtained from $\Gamma$ by deleting vertex $n$ and the corresponding line. Clearly, we can get a basis of $Z(\Gamma)$ by add elements $\{e^n, c^n, \alpha^{n-1, n}, \alpha^{n, n-1}\}$ to a basis of $Z(\Gamma\setminus \{n\})$. Furthermore, every derivation of $Z(\Gamma)$ can be obtained from one of $Z(\Gamma\setminus \{n\})$ by determining the images of the above exceptional basis by Lemma \ref{3.1},
\begin{enumerate}
\item[{\rm(1)}] $\Theta(e^n)=-t_{n, n-1}^n\alpha^{n, n-1}-t_{n-1, n}^n\alpha^{n-1, n}$;

\item[{\rm(2)}] $\Theta(c^n)=(t_{n-1, n}^{n-1, n}+t_{n, n-1}^{n, n-1})c^{n-1}$;

\item[{\rm(3)}] $\Theta(\alpha^{n-1, n})=t_{n-1, n}^{n-1, n}\alpha^{n-1, n}+t_{n, n-1}^nc^{n-1}+t_{n, n-1}^{n-1}c^n;$

\item[{\rm(4)}] $\Theta(\alpha^{n, n-1})=t_{n, n-1}^{n, n-1}\alpha^{n, n-1}+t_{n-1, n}^{n-1}c^{n-1}-t_{n-1, n}^{n-1}c^n,$
\end{enumerate}
where $$t_{n, n-1}^n=-t_{n, n-1}^{n-1},\,\,\, t_{n-1, n}^n=-t_{n-1, n}^{n-1},\eqno(3.9)$$
and change the image of $e^{n-1}$ as follows
$$\Theta(e^{n-1})=t_{n-2, n-1}^{n-1}\alpha^{n-2, n-1}+t_{n, n-1}^{n-1}\alpha^{n, n-1}+t_{n-1, n-2}^{n-1}\alpha^{n-1, n-2}+t_{n-1, n}^{n-1}\alpha^{n-1, n}.$$
Still by Lemma \ref{3.1}, we have
$$\Theta(c^{n-1})=(t_{n-2, n-1}^{n-2, n-1}+t_{n-1, n-2}^{n-1, n-2})c^{n-1}=(t_{n-1, n}^{n-1, n}+t_{n, n-1}^{n, n-1})c^{n-1},$$ that is,
$$(t_{n-2, n-1}^{n-2, n-1}+t_{n-1, n-2}^{n-1, n-2})=(t_{n-1, n}^{n-1, n}+t_{n, n-1}^{n, n-1}). \eqno(3.10)$$
Consider $(t_{n-2, n-1}^{n-2, n-1}+t_{n-1, n-2}^{n-1, n-2})$ as a constant and combine (3.9) with (3.10)
as a system of linear equations. Then it is easy to check that the solution set is a 3 dimensional manifold. Consequently,
$$\dim Der(Z(\Gamma))-\dim Der(Z(\Gamma\setminus n))=3.$$
That is, the lemma holds for $|\Gamma_0|=n$. This completes the proof.
\end{proof}

Given $x\in \mathcal{A}$, define a linear mapping $\Theta(a)=[x, a]$ for all $a\in \mathcal{A}$, where $[x, a]=xa-ax$.
Then $\Theta$ is a derivation of $\mathcal{A}$, which is called an inner derivation.
Clearly, if $x\in C(\mathcal{A})$, then the inner derivation defined by $x$ is zero. Write the linear space spanned by all inner derivations by $IDer(\mathcal{A})$.
Then if $\mathcal{A}$ is a finite dimensional algebra, we have $$\dim IDer({\mathcal{A}})=\dim \mathcal{A}-\dim C(\mathcal {A}).\eqno(3.10)$$

Combining Lemma \ref{2.1} with (3.10) leads to the following result.

\begin{lemma}\label{3.4}
$\dim IDer (Z(\Gamma))=|Q_0|+|Q_1|-1.$
\end{lemma}

Now we are in a position to give the main result of this note.

\begin{theorem}\label{3.5}
Let $\Gamma$ be a finite connected graph without loops, cycles or multiple edges and $Z(\Gamma)$ the associated zigzag algebra. If $|\Gamma_0|>1$,
then the dimension of 1th Hochschild cohomology group of $Z(\Gamma)$ is 1.
\end{theorem}

\begin{proof}
It is clear that $|Q_1|=2(|Q_0|-1)$ by the properties of $\Gamma$. Then the theorem is a direct corollary of Lemma \ref{3.3} and Lemma \ref{3.4}.
\end{proof}

Since arbitrary multiplicity one Brauer tree algebra is derived equivalent to a
zigzag algebra, we have the following obvious corollary, which can not be obtained from \cite[Theorem 4.4]{H}.
For the definition of a Brauer tree algebra, we refer the reader to \cite{ARS}.

\begin{corollary}
The dimension of the 1th Hochschild cohomology of a multiplicity one Brauer tree algebra is 1.
\end{corollary}

\begin{remarks}
(1) Since Hochschild cohomology is invariant under derived equivalence, the dimension of the
 1th Hochschild cohomology group of each algebra derived equivalent to a zigzag algebra is 1.

(2) For a zigzag algebra of type A, Mazorchuk and Stroppel \cite{MS} computed $n$-th Hochschild cohomology group for all $n\geq 0$.
\end{remarks}

\bigskip


\end{document}